\documentclass[11pt]{article}
\usepackage{epic,latexsym,amssymb}
\usepackage{color}
\usepackage{tikz}
\usepackage{amsmath,amsthm}
\usepackage{dirtytalk}

\newtheorem{thm}{Theorem}
\newtheorem{lem}{Lemma}
\newtheorem{quest}{Question}

\newtheorem{defn}{Definition}

\newcommand{\Z}{{\Z B}}

\let\oldenumerate\enumerate
\renewcommand{\enumerate}{
	\oldenumerate
	\setlength{\itemsep}{0pt}
	\setlength{\parskip}{0pt}
	\setlength{\parsep}{0pt}
}

\def\vertex(#1){\put(#1){\circle*{2}}}
\def\vertexo(#1){\put(#1){\circle{2}}}
\def\vert(#1){\put(#1){\circle*{1.5}}}
\def\verto(#1){\put(#1){\circle{1.5}}}
\def\lab(#1)#2{\put(#1){\makebox(0,0)[c]{#2}}}

\begin{document}

	\title{ The $1$-nearly edge independence number of a graph}
	
	\author{Zekhaya B. Shozi \thanks{Research supported by University of KwaZulu-Natal.}\\
		School of Mathematics, Statistics \& Computer Science\\
		University of KwaZulu-Natal\\
		Durban, 4000 South Africa\\
		\small \tt Email: zekhaya@aims.ac.za
	}

	\date{}
	\maketitle

	\begin{abstract}
		Let $G = (V(G), E(G))$ be a graph. The maximum cardinality of a set $M_k \subseteq E(G)$ such that $M_k$ contains exactly $k$-pairs of adjacent edges of $G$ is called the $k$-nearly edge independence number of $G$, and is denoted by $\alpha'_k(G)$. In this paper we study $\alpha_1'(G)$. In particular, we prove a tight lower (resp. upper) bound on $\alpha_1(G)$ if $G$ is a graph with given number of vertices. Furthermore, we present a characterisation of the general (resp. connected) graphs with given number of vertices and smallest $1$-nearly edge independence number. Lastly, we pose an open problem for further exploration of this study.
	\end{abstract}

	\newpage
	
	\section{Introduction}
	
	A \emph{simple} and \emph{undirected graph} $G$ is an ordered pair of sets $(V(G), E(G))$, where $V(G)$ is a nonempty set of elements called \emph{vertices} and $E(G)$ is a (possibly empty) set of $2$-element subsets of $V(G)$ called \emph{edges}. For convenience, we will often write $uv$ instead of $\{u,v\}$ to represent the edge joining the vertices $u$ and $v$ in a graph $G$. The number of vertices of a graph $G$ is the \emph{order} of $G$, while the number of edges of $G$ is the \emph{size} of $G$. 
	
	A set $M \subseteq E(G)$ is an \emph{edge independent set} (or a \emph{matching}) of $G$ if no two edges in $M$ are adjacent in $G$; that is, if no two edges in $M$ share a common vertex in $G$. The maximum cardinality of $M$ in $G$ is called the \emph{edge independence number} (or the \emph{matching number}) of $G$, and is denoted by $\alpha'(G)$. The edge independence number in graphs has been studied in the literature. See, for example, the classical books by Lov\'{a}sz and Plummer \cite{lovasz1986plummer}. Also, Plummer \cite{plummer20035} and Pulleyblank \cite{pulleyblank1995matchings} gave excellent survey articles on this study. In a series of papers \cite{furst2019uniquely, haxell2017lower, henning2012independent, henning2014induced, henning2023generalization, henning2007tight, henning2018tight}, lower bounds on the edge independence number in graphs have been studied. 
	
	In this paper we propose a generalisation of the edge independence number. Let $G = (V(E), E(G))$ be a graph. For an integer $k\ge 0$, we define a \emph{$k$-nearly edge independent set} of $G$ to be a set $M_k \subseteq E(G)$ such that $M_k$ contains exactly $k$ pairs of adjacent edges of $G$. The maximum cardinality of $M_k$ in $G$ is called the \emph{$k$-nearly edge independence number} of $G$, and is denoted by $\alpha_k'(G)$. We remark that a $0$-nearly edge independent set $M_0$ of a graph $G$ is simply a matching of $G$, and therefore $\alpha_0'(G)$ is the matching number of $G$. 
	
	In this paper we study $\alpha_1'$. For the classes of graphs that we investigated, the behavior of $\alpha_1'$ is sometimes similar and sometimes different to that of $\alpha_0'$. It is well-known that among all graphs with $n$ vertices, the edgeless graph $\overline{K_n}$ has the smallest $\alpha_0'$, since $\alpha_0'(\overline{K_n})=0$. However, the graph with $n$ vertices and smallest $\alpha_1'$ is the graph $G = tK_2 \cup (n-2t)K_1$, where $t \ge 0$ is an integer. This is interesting because the graph $\frac{n}{2}K_2$ is known to have the largest $\alpha_0'$, yet it proves itself to have the smallest $\alpha_1'$. Also, while the star $K_{1,n-1}$ is the only tree with the smallest $\alpha_0'$ among all trees with $n$ vertices, it is not the only tree with the smallest $\alpha_1'$, since the path $P_4$ also attains the smallest $\alpha_1'$.
	
	The rest of the paper is structured as follows. We begin with the preliminary section, where we provide some terminology and notation that we will adhere to throughout the paper. We also establish, in Section \ref{subsec:relation-between-alpha-1-and-alpha-1-prime}, a relationship between the $1$-nearly edge independence number and the $1$-nearly vertex independence number. Furthermore, in Section \ref{explicit-formulas} we provide explicit formulas for the $1$-nearly edge independence number of some classes of graphs. Our main results are in Section \ref{Sec:Main}, where we prove a tight lower (resp. upper) bound on the $1$-nearly edge independence number of a general graph with given order. The general graph with given order and smallest $1$-nearly edge independence number is fully characterised. We also prove, in Section \ref{subsec:lower-bound-for-connected-graphs} we prove a tight lower bound on the $1$-nearly edge independence number of a connected graph with given order. Again, the family of the connected graphs with given order and smallest $1$-nearly edge independence number is fully characterised. We conclude this paper in Section \ref{sec:open-question}, where we pose an open question for further exploration of this study.

	\section{Preliminary}
	\label{preliminary}
	
	For graph theory notation and terminology, we generally follow~\cite{henning2013total}. Let $G$ be a graph with vertex set $V(G)$, edge set $E(G)$, order $n = |V(G)|$ and size $m = |E(G)|$. We denote the degree of a vertex $v$ in $G$ by $\deg_Gv$, and we denote maximum degree of $G$ by $\Delta(G)$. We denote by $\overline{G}$  the \emph{complement} of $G$, which is defined as 
	$$\overline{G}=(V(G),\{uv \mid u,v\in V(G), u\neq v \text{ and } uv\notin E(G)\}).$$
	We define the \emph{open neighbourhood} of a vertex $v$ of a graph $G$ to be the set
	$$N_G(v) = \{ u \in V(G) \mid uv \in E(G) \}.$$

	For a subset $S$ of edges of a graph $G$, we denote by $G - S$ the graph obtained from $G$ by deleting the edges in $S$. If $S = \{e\}$, then we simply write $G - e$ rather than $G - \{e\}$.  A \emph{diamond} $D_4$ is the graph $K_4-e$, where $e$ is an arbitrary edge of $K_4$. The \emph{union} of two graphs $G$ and $H$ is given by
	$$
	G\cup H=(V(G)\cup V(H), E(G)\cup E(H)).
	$$
	For an integer $t\ge 0$, we write $tG$ to represent the union of the $t$ copies of a graph $G$.  We denote by $G[S]$ the subgraph of $G$ induced by the set $S \subseteq E(G)$.
	
	We denote the path graph, the cycle graph and the complete graph on $n$ vertices by $P_n$, $C_n$, and $K_n$, respectively. For positive integers $r$ and $s$, we denote by $K_{r,s}$ the complete bipartite graph with partite sets $X$ and $Y$ such that $|X|=r$ and $|Y|=s$. A complete bipartite graph $K_{1,n-1}$ is also called a \emph{star} in the literature.  Let $U_{1,n-1}$ be the graph obtained from the star $K_{1,n-1}$ by adding one more edge to make it a unicyclic graph.
	
	We use the standard notation $[k] = \{1,\ldots,k\}$.
	
	\subsection{Relation between $\alpha_1$ and $\alpha_1'$}
	\label{subsec:relation-between-alpha-1-and-alpha-1-prime}
	
	\begin{defn}
		\label{defn:line-graph}
		The line graph $L(G)$ of a graph $G$ is the graph with set of vertices $E(G)$, and such that two different elements  $e$ and $e'$ of $E(G)$ are adjacent in $L(G)$ if they have a common end in $G$.
	\end{defn}
	
	We now present the following lemma.
	\begin{lem}
		\label{Lem:GtoLG}
		For any graph $G$, we have 
		$\alpha_1'(G)=\alpha_1(L(G))$.
	\end{lem}
	
	\begin{proof}
		Let $G$ be a graph with vertex set $V(G) = \{v_1, v_2, \ldots, v_n\}$ and edge set $E(G) = \{e_1, e_2, \ldots, e_m\}$. Let $M_1(G)$ be a $1$-nearly edge independent set of $G$, and let $G[M_1(G)]$ be the subgraph of $G$ induced by $M_1(G)$. Then, $G[M_1(G)] \cong K_{1,2} \cup tK_2$, where
		\begin{align*}
			0\le t \le \begin{cases}
				\frac{n-3}{2} & \text{ if } n \text{ is odd},\\
				\frac{n-4}{2} & \text{ if } n \text{ is even}.
			\end{cases}
		\end{align*}
		Thus, $G[M_1(G)]$ has size $t+2$. 
		
		Now let $L(G)$ be the line graph of $G$. Then, by Definition \ref{defn:line-graph}, $V(L(G)) = \{e_1, e_2, \ldots, e_m\}$. Let $I_1(L(G))$ be a $1$-nearly vertex independent set of $G$, and let $G[I_1(L(G))]$ be the subgraph of $G$ induced by $I_1(L(G))$. Then, $G[I_1(L(G))] \cong K_2 \cup tK_1$, where $0\le t \le n-2$. Thus, $G[I_1(L(G))]$ has order $t+2$. 
		
		Since, by Definition \ref{defn:line-graph}, $m(G) = n(L(G))$, we have $|M_1(G)| = |I_1(L(G))|$, and consequently $\alpha_1'(G) =  \alpha_1(L(G))$.
	\end{proof}
	
	\subsection{Explicit formulas for $\alpha_1$ of some graphs}
	\label{explicit-formulas}
	
	In this subsection we provide some explicit formulas for the $1$-nearly edge independence number of some classes of graphs. We particularly consider the classes of graphs that have well-known line graphs, such as the star $K_{1,n-1}$, the path graph $P_n$ and the cycle graph $C_n$. 
	
	The following results have been recently established \cite{shozi20241}.
	\begin{itemize}
		\item $\alpha_1(K_n) = 2,$
		\item $\alpha_1(P_n) = \left \lfloor \frac{n+2}{2} \right \rfloor$,
		\item $\alpha_1(C_n) = \left \lfloor \frac{n+1}{2} \right \rfloor$.
	\end{itemize}
	
	Thus, by Lemma \ref{Lem:GtoLG}, we also have
	\begin{itemize}
		\item $\alpha_1'(K_{1,n-1}) = \alpha_1(L(K_{1,n-1})) =\alpha_1(K_{n-1}) = 2$,
		\item $\alpha_1'(P_n) = \alpha_1(L(P_n)) = \alpha_1(P_{n-1}) = \left \lfloor \frac{n+1}{2} \right \rfloor$,
		\item $\alpha_1'(C_n) = \alpha_1(L(C_n)) = \alpha_1(C_n) = \left \lfloor \frac{n+1}{2} \right \rfloor$.
	\end{itemize}
	
	\section{Main result}
	\label{Sec:Main}
	
	In this section we present a sharp lower (resp. upper) bound on $\alpha_1'(G)$ if $G$ is a general  graph of order $n$. The general graph with given number of vertices and smallest $1$-nearly edge independence number is fully characterised. Furthermore, we present a sharp lower bound on $\alpha_1'(G)$ is $G$ is a connected graph of order $n$. The family of the connected graphs with given number of vertices and smallest $1$-nearly edge independence number is fully characterised.

	\subsection{Lower bound for general graphs}
	
	In this subsection we present a tight lower bound on the $1$-nearly edge independence number of a graph in terms of its order. Furthermore, we characterise the extremal graph with given order and smallest 1-nearly edge independence number.
	
	\begin{thm}
		\label{thm:lower-bound-general-graphs}
		If $G$ is a graph of order $n$, then $\alpha_1'(G) \ge 0$, with equality if and only if $G \cong tK_2 \cup (n-2t)K_1$ for some integer $t\ge 0$.
	\end{thm}
	
	\begin{proof}
		Let $G$ be a graph of order $n$. For any integer $t\ge 0$, if $G \cong tK_2 \cup (n-2t)K_1$, then $\alpha_1'(G) =0$, thereby proving the lower bound. We now show that if $G$ is any graph of order $n$ such that $\alpha_1'(G) =0$, then  $G \cong tK_2 \cup (n-2t)K_1$. It suffices to show that $\Delta(G) \le 1$. Suppose, to the contrary, that $\Delta(G)\ge 2$. Let $v$ be a vertex of maximum degree in $G$, and let $\{v_1, v_2\} \subseteq N_G(v)$. Thus, the edges $vv_i$ for $i\in [2]$ form a $1$-nearly edge independent set in $G$, implying that
		\begin{align*}
			\alpha_1'(G) \ge |\{vv_1, vv_2\}| =2 >0,
		\end{align*}
		a contradiction. This completes the proof of \ref{thm:lower-bound-general-graphs}.  
	\end{proof}
	
	\subsection{Upper bound for general graphs}
	
	In this subsection we present a sharp upper bound on the $1$-nearly edge independence number of a graph with given order. We also characterise the family of graphs of with given order and largest $1$-nearly edge independence number.   
	
	\begin{defn}
		\label{defn:maximal-graphs}
		Let $G$ be a graph of order $n$. Then $G$ is a faithful graph if there exists a spanning forest $F$ in $G$ such  that
		\begin{align*}
			F = \begin{cases}
				P_3\cup \left( \frac{n-3}{2} \right)K_2 &\text{ if } n \text{ is odd},\\\\
				P_3\cup \left( \frac{n-4}{2} \right)K_2\cup K_1 &\text{ if } n \text{ is even}.
			\end{cases}
		\end{align*}
		Let $\mathcal{F} = \{G \mid G \text{ is a faithful graph}\}$.
	\end{defn} 
	
	It follows from the definition of a faithful graph that a spanning forest $F$ of a faithful graph $G$ has size
	\begin{align*}
		m(F) = \begin{cases}
			\frac{n+1}{2} &\text{ if } n \text{ is odd},\\\\
			\frac{n}{2} &\text{ if } n \text{ is even}.
		\end{cases}
	\end{align*}
	
	\begin{thm}
		\label{thm:upper-bound-general-graphs}
		If $G$ is a graph of order $n$, then $\alpha_1'(G) \le \left \lfloor \frac{n+1}{2} \right \rfloor$, with equality if and only if $G \in \mathcal{F}$.
	\end{thm}
	
	\begin{proof}
		Let $G$ be a graph of order $n$. If $n$ is odd and $G$ contains a spanning path $P_n$, then $\alpha_1'(G) = \frac{n+1}{2}$, while if $n$ is even and $G$ contains a spanning path $P_n$, then $\alpha_1'(G) = \frac{n}{2}$. This proves the upper bound.
		
		Let $G$ be a graph of order $n$ having the largest $\alpha_1(G)$. Suppose, to the contrary, that $G \notin \mathcal{F}$. Let $F_1$ be a spanning forest of $G$. Then, $F_1$ has size
		\begin{align*}
			m(F_1) \le \begin{cases}
				\frac{n+1}{2}-1 &\text{ if } n \text{ is odd},\\\\
				\frac{n}{2}-1 &\text{ if } n \text{ is even}.
			\end{cases}
		\end{align*}
		Thus, if $n$ is odd, then $\alpha_1'(G) \le m(F_1) = \frac{n+1}{2}-1 < \frac{n+1}{2}$, a contradiction. Similarly, if $n$ is even, then $\alpha_1'(G) \le m(F_1) = \frac{n}{2}-1 < \frac{n}{2}$, a contradiction.
	\end{proof}
	
	Definition \ref{defn:maximal-graphs} already provides a full characterisation of the family $\mathcal{F}$. However, using that description, it remains a challenge to imagine how the structure of an element of $\mathcal{F}$ should look like, especially when the number of vertices becomes large. It is worth noting, though, that $\{P_n, C_n, K_n\} \subset \mathcal{F}$. In Figure \ref{fig:graphs-in-the-family-F} we give just a few examples of graphs with given number of vertices and largest $1$-nearly edge independence number.

	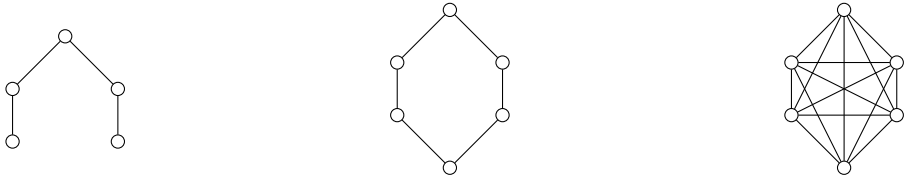
\begin{figure}[!h]
		\centering
		\begin{minipage}{0.33\textwidth}
			\centering 
			\begin{tikzpicture}[scale=0.7]
				\tikzstyle{vertexX}=[circle,draw, fill=white!90, minimum size=10pt, 
				scale=0.5, inner sep=0.2pt]
				%defining the vertices
				\node (v_1) at (0,0) [vertexX, , label=left:] {};
				\node (v_2) at (1,-1) [vertexX , label=left:] {};
				\node (v_3) at (1,-2) [vertexX, , label=left:] {};
				\node (v_4) at (-1,-2) [vertexX, , label=right:] {};
				\node (v_5) at (-1,-1) [vertexX, , label=left:] {};
				%defining the edges
				
				\draw (v_4) -- (v_5);
				\draw (v_5) -- (v_1);
				\draw (v_1) -- (v_2);
				\draw (v_2) -- (v_3);
				%\draw (v_5) -- (v_2);
				%\draw (v_1) -- (v_5);
			\end{tikzpicture}
		\end{minipage}%
		\begin{minipage}{0.33\textwidth}
			\centering
			
			\begin{tikzpicture}[scale=0.7]
				\tikzstyle{vertexX}=[circle,draw, fill=white!90, minimum size=10pt, 
				scale=0.5, inner sep=0.2pt]
				%defining the vertices
				\node (v_1) at (0,0) [vertexX, , label=left:] {};
				\node (v_2) at (1,-1) [vertexX , label=left:] {};
				\node (v_3) at (1,-2) [vertexX, , label=left:] {};
				\node (v_4) at (0,-3) [vertexX, , label=right:] {};
				\node (v_5) at (-1,-2) [vertexX, , label=left:] {};
				\node (v_6) at (-1,-1) [vertexX, , label=left:] {};
				%defining the edges
				
				\draw (v_4) -- (v_5);
				\draw (v_5) -- (v_6);
				\draw (v_6) -- (v_1);
				\draw (v_1) -- (v_2);
				\draw (v_2) -- (v_3);
				\draw (v_3) -- (v_4);
			\end{tikzpicture}
		\end{minipage}
		\begin{minipage}{0.33\textwidth}
			\centering
			
			\begin{tikzpicture}[scale=0.7]
				\tikzstyle{vertexX}=[circle,draw, fill=white!90, minimum size=10pt, 
				scale=0.5, inner sep=0.2pt]
				%defining the vertices
				\node (v_1) at (0,0) [vertexX, , label=left:] {};
				\node (v_2) at (1,-1) [vertexX , label=left:] {};
				\node (v_3) at (1,-2) [vertexX, , label=left:] {};
				\node (v_4) at (0,-3) [vertexX, , label=right:] {};
				\node (v_5) at (-1,-2) [vertexX, , label=left:] {};
				\node (v_6) at (-1,-1) [vertexX, , label=left:] {};
				%defining the edges
				
				\draw (v_1) -- (v_2);
				\draw (v_1) -- (v_3);
				\draw (v_1) -- (v_4);
				\draw (v_1) -- (v_5);
				\draw (v_1) -- (v_6);
				\draw (v_2) -- (v_3);
				\draw (v_2) -- (v_4);
				\draw (v_2) -- (v_5);
				\draw (v_2) -- (v_6);
				\draw (v_3) -- (v_4);
				\draw (v_3) -- (v_5);
				\draw (v_3) -- (v_6);
				\draw (v_4) -- (v_5);
				\draw (v_4) -- (v_6);
				\draw (v_5) -- (v_6);
			\end{tikzpicture}
		\end{minipage}
		\caption{Some examples of graphs in the family $\mathcal{F}$.}
		\label{fig:graphs-in-the-family-F}
	\end{figure}

	\subsection{Lower bound for connected graphs}
	\label{subsec:lower-bound-for-connected-graphs}
	
	We now impose a minor restriction on the problem and consider only connected graphs with given order $n$, where $n\ge 3$. Furthermore, we provide a full characterisation of the extremal connected graphs with given number of vertices and smallest possible $1$-nearly edge independence number. First, we prove one important lemma which we shall refer to later.
	
	\begin{lem}
		\label{lem:G-contains-a-cycle-implies-that-G-has-at-most-4-vertices}
		Let $G$ be a connected graph that contains a cycle. If $\alpha_1'(G) =2$, then $G$ contains at most four vertices.
	\end{lem}
	
	\begin{proof}
		Let $G$ be a connected graph of order $n$ that contains a cycle and satisfying $\alpha_1'(G)=2$. Suppose, to the contrary, that $n\ge 5$. Let $C_r= v_1, v_2, \ldots, v_r, v_1$ be a largest cycle of $G$. If $r\ge 5$, then the set $M_1(G) = \{v_1v_2, v_2v_3, v_4v_5\}$ forms a $1$-nearly edge independent set in $G$, implying that $\alpha_1'(G) \ge |M_1(G)| = 3>2$, a contradiction. Hence, we may assume that $r \le 4$.

		Suppose $r=4$. If $G = C_r$, then $G = C_4$. Hence, we may assume that $G\ne C_r$. Thus, $\deg_Gv_i \ge 3$ for some $i \in [r]$. Without any loss of generality, let $\deg_G v_1 \ge 3$. Since $n\ge 5$, there exists a vertex $v_1^* \in N_G(v_1)$ such that $v_1^*$ does not lie on $C_r$. However, the set $M_1(G) = \{v_1^*v_1, v_1v_2, v_3v_4\}$ forms a $1$-nearly edge independent set in $G$, implying that $\alpha_1'(G) \ge |M_1(G)| = 3>2$, a contradiction. Hence, we may assume that $r=3$.
		
		Suppose $r=3$. If $G = C_r$, then $G = C_3$.  Hence, we may assume that $G\ne C_r$. Thus, $\deg_Gv_i \ge 3$ for some $i \in [r]$.  Without any loss of generality, let $\deg_G v_i \ge 3$ for $i\in [2]$. Since $n\ge 5$, there exists at least two vertices $v_1^* \in N_G(v_1)$ and $v_2^*$ such that for each $i\in [2]$, $v_i^*$ does not lie on $C_r$. If $v_2^* \in N_G(v_1^*)$, then the set $M_1(G) = \{v_2^*v_1^*, v_1^*v_1, v_2v_3\}$ forms a $1$-nearly edge independent set in $G$, implying that $\alpha_1'(G) \ge |M_1(G)| = 3>2$, a contradiction. If $v_2^* \notin N_G(v_1^*)$, then $v_2^* \in N_G(v_2)$. However, the set $M_1(G) = \{v_1^*v_1, v_1v_3, v_2v_2^*\}$ forms a $1$-nearly edge independent set in $G$, implying that $\alpha_1'(G) \ge |M_1(G)| = 3>2$, a contradiction. This completes the proof of Lemma \ref{lem:G-contains-a-cycle-implies-that-G-has-at-most-4-vertices}.
	\end{proof}

	\begin{thm}
		\label{thm:lower-bound-for-connected-graphs}
		Let $n\ge 3$ be an integer. If $G$ is a connected graph of order $n$, then $\alpha_1'(G) \ge 2$, with equality if and only if $G \in \{C_3, C_4, D_4, K_4, U_{1,3}, P_4, K_{1,n-1}\}$.
	\end{thm}
	
	\begin{proof}
		Let $n\ge 3$ be an integer, and let $G$ be a connected graph of order $n$. Since $G$ is connected and $n\ge 3$, $\Delta(G)\ge 2$. Let $v$ be a vertex of maximum degree in $G$, and let $\{v_1, v_2\} \subseteq N_G(v)$. Thus, the edges $vv_i$ for all $i\in [2]$ form a $1$-nearly edge independent set in $G$, implying that
		\begin{align*}
			\alpha_1'(G) \ge |\{vv_1, vv_2\}| =2,
		\end{align*}
		thereby proving the lower bound.
		
		Suppose that $G$ is a connected graph of order $n\ge 3$ and $\alpha_1'(G) =2$. If $G$ contains a cycle, then, by Lemma \ref{lem:G-contains-a-cycle-implies-that-G-has-at-most-4-vertices}, we have $n\le 4$, implying that $G \in \{C_3, C_4, D_4, K_4, U_{1,3}\}$. Hence, we may assume that $G$ does not contain a cycle. If $G$ does not contain a cycle, then $G$ is a tree. If $n=3$, then $G \cong K_{1,2}$ and $\alpha_1'(G) =2$. Hence, we may assume that $n\ge 4$. 
		
		We show that either $\Delta(G)=2$ or $\Delta(G) =n-1$. Suppose, to the contrary, that $3 \le \Delta(G) \le n-2$. If $n=4$, then either $G \cong K_{1,3}$ in which case $\Delta(G)=3=n-1$ or $G \cong P_4$ in which case $\Delta(G)=2<3$. Thus, in both cases, we reach a contradiction. Hence we may assume that $n\ge 5$. Let $v$ be a vertex of maximum degree in $G$ and let $N_G(v) = \{v_1, v_2, \ldots, v_r\}$, where $3\le r = \Delta(G)\le n-2$. Since $G$ has order $n$ and $G$ is connected, there exists at least one neighbour of $v$ that has degree at least $2$ in $G$. Without any loss of generality, let $\deg_Gv_1\ge 2$ and let $v_1^*$ be a neighbour of $v_1$ such that $v_1^* \ne v$. Thus, the set $M_1(G) = \{vv_2, vv_3, v_1v_1^*\}$ forms a $1$-nearly edge independent set in $G$, implying that $\alpha_1'(G) \ge |M_1(G)| = 3>2$, a contradiction.
		
		Thus, either $\Delta(G) =2$ or $\Delta(G)=n-1$. If $\Delta(G)=n-1$, then, since $G$ is a tree, $G\cong K_{1,n-1}$. Hence, we may assume that $\Delta(G) = 2$. If $\Delta(G)=2$, then, since $G$ is a tree, $G$ is a path $P_n = v_1, v_2, \ldots, v_n$. We show that $n\le 4$. Suppose, to the contrary, that $n\ge 5$. Then, the set $M_1(G) = \{v_1v_2, v_2v_3, v_4v_5\}$ forms a $1$-nearly edge independent set in $G$, implying that $\alpha_1'(G) \ge |M_1(G)| = 3>2$, a contradiction, and so $n\le 4$. If $n=4$, then $G \cong P_4$, and if $n=3$, then $G \cong K_{1,2}$. This completes the proof of Theorem \ref{thm:lower-bound-for-connected-graphs}.
	\end{proof}

	\section{Open question}
	\label{sec:open-question}
	
	We have seen from the proof of Theorem \ref{thm:lower-bound-for-connected-graphs} that if $G$ is a graph of order $n$ with minimum $1$-nearly edge independence number, then either $\Delta(G) =2$ or $\Delta(G) = n-1$. It is, therefore, a natural question to find a lower bound on $\alpha_1'(G)$ if $G$ is a graph of order $n$ such that $3\le \Delta(G) \le n-2$. So, we conclude this paper by posing the following question.
	
	\begin{quest}
		Can we find a tight lower bound on $\alpha_1'(G)$ if $G$ is a graph of order $n$ and $3\le \Delta(G) \le n-2$? Can we characterise all the extremal graphs achieving equality on this bound?
	\end{quest}
	
	\newpage
	\bibliographystyle{abbrv} 
	\bibliography{references}

\end{document}